\documentclass{amsproc}

\usepackage{amsmath}
\usepackage{amssymb}
\usepackage{rotating}
\usepackage{graphicx}
\usepackage[tableposition=below]{caption}
\usepackage{bm}
\DeclareGraphicsExtensions{.png,.pdf,.eps}
\usepackage[all,2cell,cmtip]{xy}
\usepackage{tikz}
\usepackage{xcolor}
\usepackage{longtable}
\usepackage{soul}
\newtheorem{theorem}{Theorem}[section]
\newtheorem{lemma}[theorem]{Lemma}
\newtheorem{corollary}[theorem]{Corollary}
\newtheorem{proposition}[theorem]{Proposition}

\theoremstyle{definition}

\newtheorem{definition}[theorem]{Definition}

\newtheorem{remark}[theorem]{Remark}
\newtheorem{setting}[theorem]{Setting}

\def\P{{\mathbb P}}
\def\Q{{\mathbb Q}}

\def\Z{{\mathbb Z}}

\def\cE{{\mathcal E}}
\def\cF{{\mathcal F}}

\def\cI{{\mathcal I}}

\def\cL{{\mathcal L}}

\def\cO{{\mathcal O}}

\def\tensor{\otimes}

\def\cOperatorname#1{\mathop{\rm #1}\nolimits}
\def\Pic{\cOperatorname{Pic}}

\def\deg{\cOperatorname{deg}}
\def\det{\cOperatorname{det}}

\begin{document}

\title[Weak Fano bundles of rank $2$ on $Q^4$]
{Rank-two weak Fano bundles on a four-dimensional \\
quadric hypersurface $Q^4$}

\author{Yuta Takahashi}
\date{\today}
\address{Department of Mathematics, Faculty of Science and Engineering, Chuo University.\newline
1-13-27 Kasuga, Bunkyo-ku, Tokyo 112-8551, Japan}
\email{yuta0630takahashi0302@gmail.com}

\begin{abstract}
We classify rank $2$ weak Fano bundles on a four-dimensional smooth quadric hypersurface $Q^4$. Up to twisting with a line bundle, such a bundle is either a split bundle, a spinor bundle, or one of the stable bundles with Chern classes $c_1=-1$ and $c_2=(1,1)$ constructed in \cite{APW94}.
\end{abstract}

\maketitle
\tableofcontents
\section{Introduction}

Let $X$ be a smooth projective variety over an algebraically closed field $k$ of characteristic zero. A vector bundle $\cE$ on $X$ is called a \emph{Fano bundle} if its projectivization $\P(\cE)$ is a Fano variety, equivalently if $-K_{\P(\cE)}$ is ample. This notion was introduced by Szurek and Wi\'{s}niewski in their study of ruled Fano varieties. They showed that the base of a Fano bundle is itself a Fano variety and studied rank $2$ Fano bundles on $\P^3$ and on a smooth quadric threefold $Q^3$ \cite{SW90a}. Related low-dimensional cases, including rank $2$ Fano bundles on surfaces and higher-rank Fano bundles on $\P^2$, were also investigated in \cite{SW90b,SW90c}. The classification of rank $2$ Fano bundles on $Q^3$ was completed by Sols, Szurek and Wi\'{s}niewski \cite{SSW91}. Subsequently, Ancona, Peternell and Wi\'{s}niewski classified rank $2$ Fano bundles on projective spaces and quadric hypersurfaces of dimension at least $4$ \cite{APW94}. More generally, Mu\~{n}oz, Occhetta and Sol\'{a} Conde studied rank $2$ vector bundles on smooth Fano varieties of Picard number one from the viewpoint of the nef and pseudoeffective cones of their projectivizations \cite{MOS14a}, and obtained a classification theorem for rank $2$ Fano bundles on smooth Fano varieties with $H^2(X,\Z)\simeq H^4(X,\Z)\simeq \Z$ \cite{MOS14b}.

Langer introduced a weak analogue in \cite{Lan98}. A vector bundle $\cE$ is called a \emph{weak Fano bundle} if $\P(\cE)$ is a weak Fano variety, equivalently if $-K_{\P(\cE)}$ is nef and big. Rank $2$ weak Fano bundles on projective spaces were studied by Yasutake \cite{Yas12}. In dimension three, Ishikawa classified rank $2$ weak Fano bundles on cubic threefolds \cite{Ish16}. Fukuoka, Hara and Ishikawa classified rank $2$ weak Fano bundles on del Pezzo threefolds of degree $4$ in \cite{FHI22} and of degree $5$ in \cite{FHI23}. In the latter paper, they also proved that such bundles on del Pezzo threefolds of degrees one and two split, thereby completing the classification for del Pezzo threefolds of Picard number one. More recently, they obtained a classification of rank $2$ weak Fano bundles on Fano threefolds of Picard number one, including detailed resolutions in the quadric threefold case \cite{FHI25}.

For quadric hypersurfaces, the classification of rank $2$ Fano bundles is covered by \cite{SW90a,SSW91,APW94}, but the weak Fano case requires additional arguments. In \cite{Tak26}, the author classified rank $2$ weak Fano bundles on smooth quadric hypersurfaces $Q^n$ of dimension $n\ge 5$: up to twisting with a line bundle, such a bundle is either a split bundle or the Cayley bundle on $Q^5$. The four-dimensional quadric hypersurface $Q^4$ is not covered by this result. 

The main new feature is that $A^2(Q^4)\simeq \mathbb Z\alpha\oplus\mathbb Z\beta$, where \(\alpha\) and \(\beta\) denote the classes of planes belonging to the two families of planes on \(Q^4\). Therefore, the second Chern class has two independent components. Consequently, the numerical reduction in the style of \cite{APW94} leaves more cases, and several of them must be excluded by restricting the bundles to planes on $Q^4$ and using the classification of linear-uniform bundles on quadrics due to Fritzsche \cite{Fri83}. The main result of this paper is the following theorem.

\begin{theorem}\label{4}
Let $\cE$ be a rank $2$ weak Fano bundle on a smooth quadric hypersurface $Q^4$. Let \(H_{Q^4}\) denote the hyperplane class on \(Q^4\). Then, up to twisting with a line bundle, $\cE$ is one of the following:
\begin{enumerate}
\item a direct sum of two line bundles;
\item one of the two spinor bundles on $Q^4$;
\item an $H_{Q^4}$-stable bundle with $c_1(\cE)=-H_{Q^4}$ and $c_2(\cE)=\alpha+\beta$, that is, one of the bundles described in \cite[Example~2.2]{APW94}.
\end{enumerate}
\end{theorem}

We briefly outline the proof of Theorem~\ref{4}. After tensoring by a line bundle, we may assume that \(\cE\) is normalized. Thus $c_1(\cE)=c_1H_{Q^4}$ with $c_1\in\{0,-1\}$. The weak Fano condition gives the nef and big divisor
$$
\xi+\frac{4-c_1}{2}H
$$
on $\mathbb P(\cE)$, where $\xi$ denotes the tautological divisor corresponding to $\cO_{\P(\cE)}(1)$ and $H$ denotes the pullback via $\pi\colon\mathbb P(\cE)\to Q^4$ of $H_{Q^4}$. 

Applying the numerical inequalities of \cite{APW94} to the nef divisor $\xi+\frac{4-c_1}{2}H$, we have $-4\le a,b\le 12$ for $c_1(\cE)=0$ and $-3\le a,b\le 12$ for $c_1(\cE)=-H_{Q^4}$, where $c_2(\cE)=a\alpha+b\beta$. The Riemann--Roch theorem reduces the problem to a finite list of cases. Then we eliminate the remaining cases one by one. In the arguments on $Q^4$, we follow the method of \cite{APW94}, combining vanishing results and intersection computations on $\mathbb P(\mathcal E)$ with the splitting criterion established in \cite{APW94}. The remaining case with $c_1(\cE)=0$, namely $c_2(\cE)=2\alpha+2\beta$, is excluded by a vanishing theorem of de Cataldo \cite[Corollary~2.4]{deC99}.

The paper is organized as follows. In Section~2, we collect the basic definitions and preliminary results used throughout the paper, including the Chow ring of $Q^4$, weak Fano bundles, nef vector bundles, and linear-uniform bundles in \cite{Fri83}. In Section~3, we establish numerical restrictions on the second Chern classes of a normalized rank $2$ weak Fano bundle on $Q^4$, following the method of \cite{APW94}. Section~4 is devoted to several lemmas, including results on nef vector bundles on $\P^2$. In Section~5, we treat the case $c_1(\mathcal E)=0$ and prove that only split bundles occur in this case. In Section~6, we treat the case $c_1(\mathcal E)=-H_{Q^4}$, excluding all remaining numerical possibilities except split bundles, the two spinor bundles and the stable bundles with $c_2(\mathcal E)=\alpha+\beta$. Finally, in Section~7, we combine the results of Sections~5 and~6 to prove the main classification theorem.

\subsection*{Notations}

Throughout this paper, the ground field $k$ is an algebraically closed field of characteristic zero. All varieties are assumed to be defined over $k$. We use standard notation as in \cite{Har77, Laz04a, Laz04b, OSS80}. Let $X$ be an $n$-dimensional smooth projective variety, let $\cE$ be a rank $r$ vector bundle on $X$, and let $H_X$ be an ample divisor on $X$.

\begin{itemize}
\item We denote a projective \(n\)-space by \(\P^n\) and a smooth quadric hypersurface in \(\P^{n+1}\) by \(Q^n\). We denote the hyperplane classes of \(\P^n\) and \(Q^n\) by \(H_{\P^n}\) and \(H_{Q^n}\), respectively.

\item We use Grothendieck's convention and set $\mathbb P(\mathcal E):=\operatorname{Proj}_X\bigl(\operatorname{Sym}^{\bullet}\mathcal E\bigr)$. 

\item We denote the natural projection by $\pi\colon\mathbb P(\mathcal E)\longrightarrow X$ and the tautological divisor corresponding to $\cO_{\P(\cE)}(1)$ by $\xi$. We set $H:=\pi^*H_X$ on $\P(\cE)$.

\item A Cartier divisor $D$ on $X$ is called {\it nef} if for every irreducible curve $C\subset X$, the intersection number $D\cdot C\ge 0$.

\item A vector bundle $\cE$ on $X$ is called {\it nef} if the tautological line bundle $\cO_{\P(\cE)}(1)$ is nef.

\item A nef Cartier divisor $L$ on $X$ is called {\it big} if $L^{\dim X}>0$.

\item A vector bundle $\cE$ on $X$ is called {\it big} if the tautological line bundle $\cO_{\P(\cE)}(1)$ is big.

\item We say that $X$ is {\it Fano} if $-K_X$ is ample. 

\item We say that $X$ is {\it weak Fano} if $-K_X$ is nef and big. 

\item A vector bundle $\cE$ is a {\it split bundle} if $\cE$ is a direct sum of line bundles. 

\item We denote the dual vector bundle of $\cE$ by $\mathcal E^*:=\mathcal{H}om_{\mathcal O_X}(\mathcal E,\mathcal O_X)$.

\item For \(t\in\mathbb Z\), we write $\cE(t):=\cE\otimes\mathcal O_X(tH_X)$. For $t\in\mathbb Q\setminus\mathbb Z$, $\cE(t)$ denotes the corresponding formal $\mathbb Q$-twist.

\item For a rank $r$ vector bundle $\cE$, its slope with respect to $H_X$ is defined by $\mu_{H_X}(\mathcal E):=\frac{c_1(\mathcal E)\cdot H_X^{n-1}}{\operatorname{rk}(\mathcal E)}$. $\mathcal E$ is {\it $H_X$-semistable} (resp.\ {\it $H_X$-stable}) if for every nonzero proper subsheaf $0\neq\mathcal F\subsetneq\mathcal E$ with $0<\operatorname{rk}(\mathcal F)<r$, $\mu_{H_X}(\mathcal F)\leq\mu_{H_X}(\mathcal E)$ (resp.\ $\mu_{H_X}(\mathcal F)<\mu_{H_X}(\mathcal E)$).

\item We denote the group of rational equivalence classes of algebraic $i$-cycles on $X$ by $A_i(X)=A^{n-i}(X)$. We denote the Chow ring of $X$ by $A(X):=\bigoplus_i A_i(X)$.

\item We denote the $i$-th Chern class of $\cE$ by $c_i(\cE)$. Assume that $\Pic(X)\simeq \Z[H_X]$. We identify $c_1(\cE)$ with the integer $c_1$, that is, $c_1(\cE)=c_1H_X\in \Pic(X)\simeq A^1(X)$. 

\item If $r=2$, we say that $\cE$ is {\it normalized} if $c_1=0$ or $-1$.

\item We denote the $i$-th Segre class of $\cE$ by $s_i(\cE)$. It is defined by the equation $s_i(\cE):=\pi_{*}(\xi^{r-1+i})$.
\end{itemize}

\section{Preliminaries}

\subsection{Properties of weak Fano bundles on $Q^4$}

\begin{definition}\label{def}
Let $X$ be a smooth projective variety and $\mathcal E$ be a vector bundle on $X$. We say that $\mathcal E$ is a \emph{weak Fano bundle} if the projective bundle $\mathbb P(\mathcal E)$ is a weak Fano variety, i.e., $-K_{\mathbb P(\mathcal E)}$ is nef and big.
\end{definition}

For a weak Fano bundle, we have the following positivity property.
\begin{lemma}\label{lem:D-nef-big}
Let $\cE$ be a normalized rank $2$ weak Fano bundle on $Q^4$. Then $\xi+\frac{4-c_1}{2}H$ is nef and big. Moreover, $\xi+3H$ is ample. In particular, $\cE(3)$ is ample.
\end{lemma}

\begin{proof}
Since $K_{Q^4}=-4H_{Q^4}$, we have
$$
 -K_{\P(\cE)}=2\xi+(4-c_1)H = 2\left(\xi+\frac{4-c_1}{2}H\right).
$$
By Definition~\ref{def}, $-K_{\P(\cE)}$ is nef and big. Hence $\xi+\frac{4-c_1}{2}H$ is also nef and big.

It remains to show that $\xi+3H$ is ample. Set $A:=\xi+\frac{4-c_1}{2}H$. Then $\xi+3H= A+\frac{c_1+2}{2}H$. Since $c_1=0$ or $-1$, the coefficient $(c_1+2)/2$ is positive.

Let $\gamma\in \overline{\operatorname{NE}}\bigl(\mathbb P(\mathcal E)\bigr)\backslash\{0\}$. Suppose that \(H\cdot\gamma>0\). Since \(A\) is nef, $(A+\frac{c_1+2}{2} H)\cdot\gamma>0$. Suppose that \(H\cdot\gamma=0\). By \cite[Proposition~II.7.10~(b)]{Har77}, there exists \(m>0\) such that \(\xi+mH\) is very ample. Hence $0<(\xi+mH)\cdot\gamma=\xi\cdot\gamma$. We obtain
\[
\left(A+\frac{c_1+2}{2} H\right)\cdot\gamma=\xi\cdot\gamma>0.
\]

Therefore \(A+\frac{c_1+2}{2} H\) is ample by Kleiman's criterion \cite[Theorem~1.4.29]{Laz04a}. Thus $\xi+3H$ is ample. Hence $\cE(3)$ is ample.
\end{proof}

We start by recalling the structure of $A^2(Q^4)$ and the notation for $c_2(\mathcal E)$. The structure of $A^2(Q^4)$ is one of the main differences from the higher-dimensional case treated in \cite{Tak26}. We refer to \cite[Section~0]{AS92} for these standard facts about the cycle classes on a smooth $4$-dimensional quadric $Q^4$.

\begin{remark}\label{rem:chow-ring-Q4}
We recall the relevant part of the Chow ring of a smooth $4$-dimensional quadric $Q^4$. Let $\alpha,\beta\in A^2(Q^4)$ be the classes of planes belonging to the two families of planes on $Q^4$. Then
$$
A^2(Q^4)=\mathbb Z\alpha\oplus\mathbb Z\beta.
$$
With our convention, we have $H_{Q^4}^2=\alpha+\beta$. Moreover, the intersection products satisfy
$$
\alpha^2=1,\qquad\beta^2=1,\qquad\alpha\beta=0,
$$
where we identify $A^4(Q^4)$ with $\mathbb Z$ via the degree map. Equivalently, 
$$
\deg(\alpha^2)=\deg(\beta^2)=1,\qquad\deg(\alpha\beta)=0.
$$
In particular, $H_{Q^4}^4=(\alpha+\beta)^2=2$. Accordingly, throughout this paper, if \(\mathcal E\) is a vector bundle on $Q^4$ with $c_2(\mathcal E)=a\alpha+b\beta$ for $a,b\in\mathbb Z$, we simply write $c_2=(a,b)$.
\end{remark}

\subsection{Standard criteria for splitting and stability}

We recall several standard facts used repeatedly. We begin with a positivity result for nef vector bundles on $Q^4$.

\begin{proposition}\cite[Proposition~1.4~(B)]{APW94}\label{prop:APWB}
Let $\mathcal E$ be a rank $2$ nef vector bundle on $Q^4$ with $c_2=(a,b)$. Then the following inequalities hold:
$$
a\ge 0,\qquad b\ge 0, \qquad c_1^2\ge a,\qquad c_1^2\ge b,
$$
and
$$
a^2+b^2-3c_1^2(a+b)+2c_1^4\ge 0.
$$
\end{proposition}

Next, we recall some criteria for splitting and stability on $Q^4$.

\begin{proposition}\cite[Proposition~1.6]{APW94}\label{prop:APW-splitting}
Let $\cE$ be a rank $2$ vector bundle on $Q^4$ with $c_2=(a,b)$, $H^0(Q^4,\cE)\ne 0$ and $H^0(Q^4,\cE(-1))=0$. If
$$
 a<0\quad\text{or}\quad b<0\quad\text{or}\quad a=b=0,
$$
then $\cE$ is a split bundle.
\end{proposition}

\begin{proposition}\label{prop:stable-section}
Let $\cE$ be a normalized rank $2$ vector bundle on a smooth projective variety $X$ with ${\rm Pic}(X)\simeq \Z[H_X]$, where $H_X$ is an effective ample generator. Then the following hold: 
\begin{enumerate}
\item $\cE$ is $H_X$-stable if and only if $H^0(X, \cE)=0$.
\item If $c_1=0$, then $\cE$ is $H_X$-semistable if and only if $H^0(X,\cE(-1))=0$.
\end{enumerate}
\end{proposition}

\begin{proof}
This can be proved by the same argument as in \cite[Lemma~1.2.5]{OSS80}.
\end{proof}

\begin{proposition}\cite[Corollary~3.5]{MP97}\label{prop:bogo}
Let $X$ be an $n$-dimensional normal projective variety which is smooth in codimension $2$. If $\cE$ is an $H_X$-semistable torsion-free sheaf of rank $r$, then
$$  
\frac{r-1}{2r}c_1(\cE)^2\cdot H_X^{n-2}\leq c_2(\cE)\cdot H_X^{n-2}.
$$
\end{proposition}

\subsection{Properties of nef vector bundles and vanishing theorems}

Next, we prove some basic properties of nef vector bundles which will be used in the following sections.
\begin{lemma}\label{lem:nef-quotient-curve}
Let $\cF$ be a nef vector bundle on a smooth projective variety $X$. Let $C\subset X$ be a smooth irreducible curve. If $\cF|_C\twoheadrightarrow \cL$ is a quotient line bundle, then $\deg \cL\ge 0$.
\end{lemma}

\begin{proof}
Since nefness is preserved under restriction and quotients, every quotient line bundle $\cL$ of $\mathcal F|_C$ is nef. Hence $\deg \cL\geq0$.
\end{proof}

\begin{lemma}\label{lem:IZ-nef}
Let $\mathcal F$ be a nef vector bundle on a smooth projective variety $X$ and let $Z\subsetneq X$ be a closed subscheme with ideal sheaf $\mathcal I_Z$. If there is a surjection $\mathcal F\twoheadrightarrow\mathcal I_Z$, then $Z=\emptyset$.
\end{lemma}

\begin{proof}
Assume that $Z\neq\emptyset$. We can choose a smooth irreducible curve $C\subset X$ such that $C\not\subset Z$ and $C\cap Z\neq\emptyset$. Restricting the surjection $\mathcal F\twoheadrightarrow \mathcal I_Z$ to $C$, we obtain a surjection $\mathcal F|_C\twoheadrightarrow \mathcal I_Z\otimes\mathcal O_C$. Tensoring the
defining exact sequence
$$
0\to \mathcal I_Z\to \mathcal O_X\to \mathcal O_Z\to0
$$
with $\mathcal O_C$, we obtain an exact sequence
$$
\mathcal I_Z\otimes\mathcal O_C\longrightarrow\mathcal O_C\longrightarrow\mathcal O_Z\otimes\mathcal O_C\longrightarrow 0.
$$
Since $\mathcal O_Z\otimes\mathcal O_C\simeq \mathcal O_{Z\cap C}$, this becomes
$$
\mathcal I_Z\otimes\mathcal O_C\longrightarrow\mathcal O_C\longrightarrow\mathcal O_{Z\cap C}\longrightarrow 0.
$$
Hence the image of 
$$
\mathcal I_Z\otimes\mathcal O_C\longrightarrow \mathcal O_C
$$
is $\ker\bigl(\mathcal O_C\to \mathcal O_{Z\cap C}\bigr)=\mathcal I_{Z\cap C/C}$. Therefore there is a natural surjection $\mathcal I_Z\otimes\mathcal O_C\twoheadrightarrow \mathcal I_{Z\cap C/C}$. Combining these two surjections, we obtain a quotient $\mathcal F|_C\twoheadrightarrow \mathcal I_{Z\cap C/C}$.

Since $C$ is a smooth curve and $C\not\subset Z$, the scheme theoretic intersection $Z\cap C$ is a zero-dimensional closed subscheme of $C$. Moreover, by our choice of $C$, it is non-empty. Thus $\mathcal I_{Z\cap C/C}\simeq \mathcal O_C(-(Z\cap C))$ is a line bundle with $\deg \mathcal I_{Z\cap C/C}=-\deg(Z\cap C)<0$. 

On the other hand, $\mathcal F|_C$ is nef, because nefness is preserved
under restriction. Since $\mathcal I_{Z\cap C/C}$ is a quotient line
bundle of $\mathcal F|_C$, Lemma~\ref{lem:nef-quotient-curve} implies
that $\deg \mathcal I_{Z\cap C/C}\ge0$, which is a contradiction. Therefore $Z=\emptyset$.
\end{proof}

In this paper, we will use the following two vanishing theorems.

\begin{theorem}\cite[Theorem~1]{LP75}\label{thm:le}
Let $X$ be a smooth projective variety and $\cE$ be an ample vector bundle of rank $r$ on $X$. Then $H^i(X, \cE\tensor K_X)=0$ for $i \ge r$.
\end{theorem}

We will also use the following vanishing theorem of de Cataldo. We first recall the notions of semiampleness and $b$-ampleness.
A line bundle $L$ on a projective variety $X$ is called \emph{semiample} if some positive tensor power of $L$ is globally generated. For a nonnegative integer $b$, a semiample line bundle $L$ on $X$ is called \emph{$b$-ample} if the morphism defined by $|L^{\otimes m}|$ has fibers of dimension at most $b$ for every sufficiently divisible positive integer $m$. A vector bundle $\cF$ on $X$ is called \emph{$b$-ample} if its tautological line bundle $\cO_{\P(\cF)}(1)$ is $b$-ample.

\begin{theorem}\cite[Corollary~2.4]{deC99}\label{thm:de-cataldo}
Let $X$ be a smooth projective variety of dimension $d$ over a field of characteristic zero, and let $\cF$ be a $b$-ample vector bundle of rank $r$ on $X$. Set $\kappa(\cF):=\kappa\bigl(\cO_{\P(\cF)}(1)\bigr)$, where the right-hand side denotes the Kodaira dimension. Then
$$
H^q(X,\Omega_X^p\tensor\cF)=0
$$
for every pair of nonnegative integers $(p,q)$ satisfying
$$
p+q>2\bigl[d+(r-1)\bigr]-\min\bigl(\kappa(\cF),d+r-b\bigr).
$$
\end{theorem}

\subsection{General facts on quadric hypersurfaces}

Finally, we recall the definition of linear-uniform bundles on $Q^n$, Fritzsche's classification results and a general fact about lines on $Q^4$.

\begin{definition}\cite[Definition~2.1]{Fri83}\label{lin}
Let $\mathcal E$ be a rank $2$ vector bundle on $Q^n$. We say that $\mathcal E$ is \emph{linear-uniform of type $\alpha$} if there exists an integer $\alpha\ge 0$ such that for every line $\ell\subset Q^n$, one has $\mathcal E|_\ell\simeq\mathcal O_\ell\oplus \mathcal O_\ell(-\alpha)$. 
\end{definition}

\begin{proposition}\cite[Theorem~3.1~(4), Lemma~3.4]{Fri83}\label{prop:Fritzsche}
Let $\mathcal E$ be a rank $2$ linear-uniform vector bundle on $Q^4$ of type $\alpha$. Then the following hold.
\begin{enumerate}
\item If $\alpha=1$, then $\cE$ is either a split bundle or, up to twisting with a line bundle, one of the two spinor bundles on $Q^4$.
\item If $\alpha\ge 2$, then $\mathcal E$ is a split bundle.
\end{enumerate}
\end{proposition}

\begin{definition}
A \emph{plane} on $Q^4$ is a linear subspace $\Pi\simeq\mathbb P^2$ with $\Pi\subset Q^4$. For a plane \(\Pi\subset Q^4\), we write $H_\Pi:=H_{Q^4}|_\Pi$, which is the hyperplane class on \(\Pi\simeq\mathbb P^2\). In particular, $A^2(\Pi)=\mathbb ZH_\Pi^2$. 

The planes on $Q^4$ form two families, whose classes in $A^2(Q^4)$ are denoted by $\alpha$ and $\beta$, respectively. A plane $\Pi\subset Q^4$ is called an \emph{$\alpha$-plane} (resp.\ a \emph{$\beta$-plane}) if $[\Pi]=\alpha$ (resp.\ $[\Pi]=\beta$) in $A^2(Q^4)$.
\end{definition}

\begin{lemma}\cite[Example~22.7 and Proposition~22.8]{Har92}\label{Hari}
Let $Q^4\subset \mathbb P^5$ be a smooth $4$-dimensional quadric. Then every line $\ell\subset Q^4$ is contained in a unique $\alpha$-plane and in a unique $\beta$-plane.
\end{lemma}

\section{Numerical restrictions on $c_2(\cE)$}

We first prove the numerical inequalities obtained from Proposition~\ref{prop:APWB}.

\begin{proposition}\label{prop:numerical-c1-0}
Let $\mathcal E$ be a rank $2$ weak Fano bundle on $Q^4$ with $c_1=0$ and $c_2=(a,b)$. Then $-4\le a,b\le 12$ and $a+b\le 8$. Moreover,
\begin{equation}\label{eq:c1-0-top}
a^2+b^2-40(a+b)+160>0.
\end{equation}
\end{proposition}

\begin{proof}
By Lemma~\ref{lem:D-nef-big}, $\xi+2H$ is nef and big. Hence $\cE(2)$ is nef and big. We have $c_1(\cE(2))=4H_{Q^4}$ and $c_2(\cE(2))=c_2(\cE)+4H_{Q^4}^2=(a+4)\alpha+(b+4)\beta$. Proposition~\ref{prop:APWB} gives $a+4\ge 0$ and $b+4\ge 0$. Hence $a,b\ge -4$.

Next, we use the positivity of the third Segre class of $\mathcal E(2)$. Since $s_3(\mathcal E(2))=c_1(\mathcal E(2))^3-2c_1(\mathcal E(2))c_2(\mathcal E(2))$, intersecting this equality with $H_{Q^4}$, we obtain
$$
s_3(\mathcal E(2))\cdot H_{Q^4}=2\cdot 4\{16-(a+b+8)\}.
$$
Since $\mathcal E(2)$ is nef, its Segre classes are nef cycles. Indeed, let $\pi\colon\mathbb P(\mathcal E(2))\to Q^4$ be the natural projection, and set $\zeta:=c_1\bigl(\mathcal O_{\mathbb P(\mathcal E(2))}(1)\bigr)$.
Then $\zeta$ is nef and
\[
s_i(\mathcal E(2))=\pi_*(\zeta^{i+1}).
\] 
Hence $s_i(\mathcal E(2))$ has non-negative intersection with every effective $i$-cycle and is therefore nef. Then $s_3(\mathcal E(2))\cdot H_{Q^4}\ge 0$. It follows that $a+b+8\le 16$. Thus $a+b\le 8$. Moreover, Proposition~\ref{prop:APWB} gives $16\ge a+4$ and $16\ge b+4$. Therefore $a\le 12$ and $b\le 12$.

Finally, since $\xi+2H$ is big, $(\xi+2H)^5>0$. For $c_1=0$, one computes
\begin{align*}
 (\xi+2H)^5 
 &= c_2(\cE)^2-40H^2c_2(\cE)+80H^4 \\
 &= a^2+b^2-40(a+b)+160.
\end{align*}
This gives \eqref{eq:c1-0-top}.
\end{proof}

\begin{proposition}\label{prop:numerical-c1-minus}
Let $\mathcal E$ be a rank $2$ weak Fano bundle on $Q^4$ with $c_1=-1$ and $c_2=(a,b)$. Then $-3\le a,b\le 12$ and $a+b\le 8$. Moreover,
\begin{equation}\label{eq:c1-minus-top}
8a^2+8b^2-324(a+b)+1441>0.
\end{equation}
\end{proposition}

\begin{proof}
By Lemma~\ref{lem:D-nef-big}, $\xi+\frac52H$ is nef and big. We apply Proposition~\ref{prop:APWB} to the formal $\mathbb Q$-twist $\cE(5/2)$; see, for example, \cite[Sections~6.2 and~8.1]{Laz04b}. It has
$$
 c_1\left(\cE\left(\frac52\right)\right)=4H_{Q^4},
$$
and
$$
 c_2\left(\cE\left(\frac52\right)\right)=\left(a+\frac{15}{4}\right)\alpha+\left(b+\frac{15}{4}\right)\beta.
$$
Proposition~\ref{prop:APWB} gives $a+\frac{15}{4}\ge 0$ and $b+\frac{15}{4}\ge 0$. Since $a$ and $b$ are integers, this gives $a,b\ge -3$. Next, we use the positivity of the third Segre class of $\mathcal E(5/2)$. Since
$$
s_3\left(\mathcal E\left(\frac52\right)\right)=c_1\left(\mathcal E\left(\frac52\right)\right)^3-2c_1\left(\mathcal E\left(\frac52\right)\right)c_2\left(\mathcal E\left(\frac52\right)\right), 
$$
intersecting this equality with $H_{Q^4}$, we get
$$
s_3\left(\mathcal E\left(\frac52\right)\right)\cdot H_{Q^4}=2\cdot4\left\{16-\left(a+b+\frac{15}{2}\right)\right\}.
$$
Since $\xi+\frac52H$ is nef, the Segre classes of the corresponding formal $\mathbb Q$-twist are nef cycles. Hence
$$
s_3\left(\mathcal E\left(\frac52\right)\right)\cdot H_{Q^4}\ge0.
$$
It follows that
$$
a+b+\frac{15}{2}\le16.
$$
Since $a+b$ is an integer, we obtain
$$
a+b\le8.
$$
Moreover, Proposition~\ref{prop:APWB} gives $16\ge a+\frac{15}{4}$ and $16\ge b+\frac{15}{4}$. Therefore $a\le 12$ and $b\le 12$.

Since $\xi+\frac52H$ is big, $(\xi+\frac52H)^5>0$. A direct Segre class computation gives
$$
 \left(\xi+\frac52H\right)^5=a^2+b^2-\frac{81}{2}(a+b)+\frac{1441}{8}.
$$
Multiplying by $8$ yields \eqref{eq:c1-minus-top}.
\end{proof}

Next, using the Riemann--Roch theorem, we restrict the values of $c_2=(a,b)$.

\begin{lemma}\cite[Section~8]{APW94}\label{lem:RR-c1-0}
Let $\cE$ be a rank $2$ vector bundle on $Q^4$ with $c_2=(a,b)$. If $c_1=0$, then
$$
\chi(Q^4, \cE)=\frac{a^2+b^2-23(a+b)+24}{12} \quad \text{and} \quad\chi(Q^4, \cE(-2))=\frac{a(a+1)+b(b+1)}{12}.
$$
In particular, $a+b\equiv 0\pmod 2$ and $a(a+1)+b(b+1)\equiv 0\pmod {12}$.
\end{lemma}

\begin{proof}
By the Riemann--Roch theorem, we obtain the formulas as in \cite[Section~8]{APW94}. For the sake of completeness, we indicate the calculation. The Todd class of $Q^4$ is
$$
 \operatorname{td}(Q^4)=1+2H_{Q^4}+\frac{23}{12}H_{Q^4}^2+\frac76H_{Q^4}^3+\frac12H_{Q^4}^4.
$$
For $c_1=0$,
$$
 \operatorname{ch}(\cE)=2-c_2(\cE)+\frac{c_2(\cE)^2}{12}.
$$
Using Remark~\ref{rem:chow-ring-Q4}, we obtain the stated formulas. 

It remains to prove the two congruences. Let $i:Q^3\hookrightarrow Q^4$ be a smooth hyperplane section, set $\cF:=i^*\mathcal E$ and let $h:=i^*H_{Q^4}$. Let $\ell\in A^2(Q^3)$ denote the class of a line. Then $c_1(\cF)=0$ and $c_2(\cF)=(a+b)\ell$. We apply the Riemann--Roch theorem to $\cF(-1)$. We have
$$
\operatorname{td}(Q^3)=1+\frac32h+\frac{13}{12}h^2+\frac12h^3.
$$
Since $\cF$ has rank $2$ and $c_1(\cF)=0$, we have
$$
\operatorname{ch}(\cF)=2-c_2(\cF)=2-(a+b)\ell.
$$
Consequently,
$$
\operatorname{ch}(\cF(-1))=2-2h+\bigl(h^2-(a+b)\ell\bigr)-\frac13h^3+(a+b)h\ell.
$$
Thus the Riemann--Roch theorem gives
$$
\chi(Q^3,\cF(-1))=-\frac{a+b}{2}.
$$
Since $\chi(Q^3,\cF(-1))$ is an integer, it follows that $a+b\equiv0\pmod 2$. Since $\chi(Q^4, \mathcal E(-2))=\frac{a(a+1)+b(b+1)}{12}$ is an integer, $a(a+1)+b(b+1)\equiv0\pmod{12}$.
\end{proof}

\begin{lemma}\cite[Section~8]{APW94}\label{lem:RR-c1-minus}
Let $\cE$ be a rank $2$ vector bundle on $Q^4$ with $c_2=(a,b)$. If $c_1=-1$, then
$$
\chi(Q^4, \cE)=\frac{a^2+b^2-13(a+b)+12}{12}.
$$
In particular, $a^2+b^2-13(a+b)\equiv 0\pmod {12}$.
\end{lemma}

\begin{proof}
Again, by the Riemann--Roch theorem, we obtain the formula as in
\cite[Section~8]{APW94}. Indeed, we use
$$
 \operatorname{ch}(\cE)=2-H_{Q^4}+\frac{H_{Q^4}^2-2c_2(\cE)}{2}+\frac{-H_{Q^4}^3+3c_2(\cE)H_{Q^4}}{6}
 +\frac{H_{Q^4}^4-4c_2(\cE)H_{Q^4}^2+2c_2(\cE)^2}{24}
$$
together with the Todd class $\operatorname{td}(Q^4)$. Since \(\chi(Q^4, \cE)\) is an integer, the stated congruence follows.
\end{proof}

Summarizing the results of this section, we obtain the following list of possible values of $c_2=(a,b)$.

\begin{proposition}\label{prop:cases}
Let $\mathcal E$ be a normalized rank $2$ weak Fano bundle on $Q^4$ with $c_2=(a,b)$. Then the possible values of $(a,b)$ are as follows.

If $c_1=0$, then
\begin{align*}
(a,b)\in
\{&(-4,-4),(-1,-1), (-4,0),(0,-4),(-3,-3),(-4,8),(8,-4),\\
&(-3,5),(5,-3),(0,0), (-1,3),(3,-1), (-3,9),(9,-3), (2,2)\}.
\end{align*}

If $c_1=-1$, then
\begin{align*}
(a,b)\in
\{&(-3,-3),(-3,0),(0,-3), (-3,1), (1,-3), (-3,4),(4,-3),(-3,9),(9,-3),\\
&(-2,-2),(-2,3), (3,-2),(0,0),(0,4),(4,0), (-2,6),(6,-2), (-2,7), (7,-2),\\
&(1,0),(0,1),(1,1)\}.
\end{align*}
\end{proposition}
\begin{proof}
For $c_1=0$, the assertion follows from
Proposition~\ref{prop:numerical-c1-0},
inequality~\eqref{eq:c1-0-top}, and
Lemma~\ref{lem:RR-c1-0}.
For $c_1=-1$, the assertion follows from
Proposition~\ref{prop:numerical-c1-minus},
inequality~\eqref{eq:c1-minus-top}, and
Lemma~\ref{lem:RR-c1-minus}.
\end{proof}

\section{Key lemmas}
In this section, we prove lemmas that will be used in the following sections.

\subsection{The Koszul sequence associated with a section}

\begin{lemma}\label{lem:koszul-section}
Let $X$ be a smooth projective variety, let $\cE$ be a rank $2$ vector bundle on $X$, and let $0\neq s\in H^0(X,\cE)$. Let $s^*\colon \cE^*\longrightarrow\mathcal O_X$ be the homomorphism dual to $s$, and let $Z:=Z(s)$ be the closed subscheme defined by the ideal sheaf $\mathcal I_Z:=\operatorname{Im}(s^*)$. Assume that $Z$ is either empty or of pure codimension $2$ in $X$. Then there is an exact sequence
$$
0\longrightarrow\mathcal O_X\xrightarrow{s}\cE\longrightarrow\mathcal I_Z\otimes\det \cE\longrightarrow0.
$$
\end{lemma}

\begin{proof}
Set $\mathcal O_Z:=\mathcal O_X/\mathcal I_Z$. We consider the complex 
$$
0\longrightarrow\bigwedge\nolimits^2 \cE^*\longrightarrow\cE^*\xrightarrow{s^*}\mathcal O_X\longrightarrow\mathcal O_Z\longrightarrow0
$$
which is the Koszul complex associated with $s^*$ as in \cite[Exercise~17.20]{Eis95}. We show that this complex is exact. First, suppose that $Z=\emptyset$. Then $s$ is nowhere vanishing. Considering the stalk at an arbitrary point $x\in X$, the section $s$ is represented by $\langle f,g\rangle\in\mathcal O_{X,x}^{\oplus2}$, where one of $f$ and $g$ is a unit. Hence the stalk of the Koszul complex at $x$ is
$$
0\longrightarrow\mathcal O_{X,x}\xrightarrow{\binom{-g}{f}}\mathcal O_{X,x}^{\oplus2}\xrightarrow{(f\ \ g)}\mathcal O_{X,x}\longrightarrow0
$$
which is exact. Since exactness can be checked on stalks, the Koszul complex is exact in this case.

Next, suppose that $Z\neq\emptyset$. Since $\operatorname{codim}_X Z=2$ and $X$ is smooth, the Koszul complex
$$
0\longrightarrow\bigwedge\nolimits^2 \cE^*\longrightarrow\cE^*\xrightarrow{s^*}\mathcal O_X\longrightarrow\mathcal O_Z\longrightarrow0
$$
is a locally free resolution of $\mathcal O_Z$ by \cite[Exercise~17.20]{Eis95}. Consequently, it is exact.

Thus the Koszul complex is exact in both cases. Since $\mathcal I_Z=\ker\bigl(\mathcal O_X\rightarrow\mathcal O_Z\bigr)=\operatorname{Im}(s^*)$, we obtain
$$
0\longrightarrow\bigwedge\nolimits^2 \cE^*\longrightarrow\cE^*\longrightarrow\mathcal I_Z\longrightarrow0.
$$
Tensoring this sequence by $\det\cE$ and using $\bigwedge\nolimits^2\cE^*\simeq(\det\cE)^{-1}$, we obtain
$$
0\longrightarrow\mathcal O_X\longrightarrow\cE^*\otimes\det \cE\longrightarrow\mathcal I_Z\otimes\det \cE\longrightarrow0.
$$

For a rank $2$ vector bundle, there is a natural isomorphism $\cE^*\otimes\det \cE\simeq \cE$. Under this isomorphism, the first homomorphism in the last exact sequence is precisely the section $s\colon\mathcal O_X\longrightarrow \cE$. Therefore,
$$
0\longrightarrow\mathcal O_X\xrightarrow{s}\cE\longrightarrow\mathcal I_Z\otimes\det \cE\longrightarrow0
$$
is exact.
\end{proof}

\subsection{Nef vector bundles on $\P^2$}

Throughout this subsection, we identify $A^2(\P^2)\simeq \mathbb Z$ via the degree map, and accordingly identify the second Chern class of a vector bundle on \(\P^2\) with its degree.
\begin{lemma}\label{lem:P2-c1-0-negative}
Let $\cF$ be a rank $2$ vector bundle on $\mathbb P^2$ with $c_1=0$. Assume that $\cF(2)$ is nef. Then the following hold.
\begin{enumerate}
\item If $c_2=-1$, then $\cF\simeq \mathcal O_{\mathbb P^2}(1)\oplus\mathcal O_{\mathbb P^2}(-1)$. 
\item If $c_2=-3$, then there is no such $\cF$.
\end{enumerate}
\end{lemma}

\begin{proof}
Since $c_2<0$ in both cases, $\cF$ is not
$H_{\P^2}$-semistable by Proposition~\ref{prop:bogo}.
By Proposition~\ref{prop:stable-section}~(ii),
$$
H^0(\mathbb P^2,\cF(-1))\neq0.
$$
Hence there exists an integer $k>0$ such that
$$
H^0(\mathbb P^2,\cF(-k))\neq0.
$$
Let $k$ be maximal with this property, and choose a nonzero section
$$
0\neq s\in H^0(\mathbb P^2,\cF(-k)).
$$
Let $Z:=Z(s)$ be its scheme theoretic zero locus. The section $s$ induces a nonzero morphism
$$
\varphi\colon\mathcal O_{\mathbb P^2}\longrightarrow \cF(-k).
$$
Since $\varphi$ is nonzero, $\ker(\varphi)$ has rank zero and is therefore a torsion subsheaf of $\mathcal O_{\mathbb P^2}$. Since $\mathcal O_{\mathbb P^2}$ is torsion-free, we obtain $\ker(\varphi)=0$. Thus $\varphi$ is injective. 

We claim that $Z$ has no one-dimensional component. Indeed, if $s$ vanished along a nonzero effective divisor $D$ of degree $m>0$, then $s$ would determine a nonzero section of $\cF(-k-m)$, and hence
$$
H^0(\mathbb P^2,\cF(-k-m))\neq0,
$$
contrary to the maximality of $k$.

Consequently, $Z$ is either empty or a zero-dimensional subscheme. Lemma~\ref{lem:koszul-section} gives an exact sequence
$$
0\longrightarrow\mathcal O_{\mathbb P^2}\xrightarrow{s}\cF(-k)\longrightarrow\mathcal I_Z\otimes\det(\cF(-k))\longrightarrow0.
$$
Since $c_1=0$, one has
$$
\det(\cF(-k))\simeq\mathcal O_{\mathbb P^2}(-2k).
$$
After tensoring by $\mathcal O_{\mathbb P^2}(k)$, we obtain
$$
0\longrightarrow\mathcal O_{\mathbb P^2}(k)\longrightarrow\cF\longrightarrow\mathcal I_Z(-k)\longrightarrow0.
$$
Computing the second Chern class, we get
$$
c_2=-k^2+\deg Z.
$$

Since $Z$ is either empty or zero-dimensional, we can choose a line $\ell\subset \mathbb P^2$ such that $\ell\cap Z=\emptyset$. Restricting the exact sequence to $\ell$, we get
$$
0\to \mathcal O_\ell(k)\to \cF|_\ell\to \mathcal O_\ell(-k)\to0.
$$
Moreover, $\operatorname{Ext}^1_\ell(\mathcal O_\ell(-k),\mathcal O_\ell(k))=H^1(\ell,\mathcal O_\ell(2k))=0$. Thus
$$
\cF|_\ell\simeq \mathcal O_\ell(k)\oplus\mathcal O_\ell(-k).
$$
Since $\cF(2)$ is nef, we have $2-k\ge0$. Thus $k\le2$. 

If $c_2=-1$, then $-1=-k^2+\deg Z$, so $\deg Z=k^2-1$. Since $k\le2$, we have $k=1$ or $k=2$. If $k=2$, then $\deg Z=3$. Tensoring the exact sequence by $\mathcal O_{\mathbb P^2}(2)$, we get a surjection $\cF(2)\twoheadrightarrow \cI_Z$. Since $Z\neq\emptyset$, this contradicts Lemma~\ref{lem:IZ-nef}. Hence $k=1$. Then $\deg Z=0$, so $Z=\emptyset$. Thus we have
$$
0\to \mathcal O_{\mathbb P^2}(1)\to \cF\to
\mathcal O_{\mathbb P^2}(-1)\to0.
$$
Since $\operatorname{Ext}^1(\mathcal O_{\mathbb P^2}(-1),\mathcal O_{\mathbb P^2}(1))=H^1(\mathbb P^2,\mathcal O_{\mathbb P^2}(2))=0$, the sequence splits. Hence $\cF\simeq \mathcal O_{\mathbb P^2}(1)\oplus\mathcal O_{\mathbb P^2}(-1)$. 

If $c_2=-3$, then $-3=-k^2+\deg Z$, so $\deg Z=k^2-3$. Since $\deg Z\ge0$ and $k\le2$, we must have $k=2$ and $\deg Z=1$. Tensoring the exact sequence by $\mathcal O_{\mathbb P^2}(2)$, we again obtain a surjection $\cF(2)\twoheadrightarrow \cI_Z$. Since $Z\neq\emptyset$, this contradicts Lemma~\ref{lem:IZ-nef}. Thus the case $c_2=-3$ does not occur.
\end{proof}

\begin{lemma}\label{lem:P2-c1-minus-negative}
Let $\cF$ be a rank $2$ vector bundle on $\mathbb P^2$ with $c_1=-1$. Assume that the formal $\mathbb Q$-twist $\cF(5/2)$ is nef. Then the following hold.
\begin{enumerate}
\item If $c_2=0$, then $\cF\simeq \mathcal O_{\mathbb P^2}\oplus\mathcal O_{\mathbb P^2}(-1)$. 
\item If $c_2=-2$, then $\cF\simeq \mathcal O_{\mathbb P^2}(1)\oplus\mathcal O_{\mathbb P^2}(-2)$. 
\end{enumerate}
\end{lemma}

\begin{proof}
Since $c_1(\mathcal F)^2>4c_2(\mathcal F)$ in both cases, $\mathcal F$ is not $H_{\P^2}$-semistable by Proposition~\ref{prop:bogo}. In particular, $\mathcal F$ is not $H_{\P^2}$-stable. By Proposition~\ref{prop:stable-section}~(i),
$$
H^0(\mathbb P^2,\mathcal F)\neq0.
$$
Hence there exists an integer $k\geq0$ such that
$$
H^0(\mathbb P^2,\mathcal F(-k))\neq0.
$$
Let $k$ be maximal with this property, and choose a nonzero section
$$
0\neq s\in H^0(\mathbb P^2,\mathcal F(-k)).
$$
Let $Z:=Z(s)$ be its scheme theoretic zero locus. The section $s$ induces a nonzero morphism
$$
\varphi\colon\mathcal O_{\mathbb P^2}\longrightarrow\mathcal F(-k).
$$
Since $\varphi$ is nonzero, $\ker(\varphi)$ has rank zero and is therefore a torsion subsheaf of $\mathcal O_{\mathbb P^2}$. Since $\mathcal O_{\mathbb P^2}$ is torsion-free, we obtain $\ker(\varphi)=0$. Thus $\varphi$ is injective.

We claim that $Z$ has no one-dimensional component. Indeed, if $s$ vanished along a nonzero effective divisor $D$ of degree $m>0$, then $s$ would determine a nonzero section of $\mathcal F(-k-m)$, and hence
$$
H^0(\mathbb P^2,\mathcal F(-k-m))\neq0,
$$
contrary to the maximality of $k$.

Consequently, $Z$ is either empty or a zero-dimensional subscheme. Lemma~\ref{lem:koszul-section} gives an exact sequence
$$
0\longrightarrow\mathcal O_{\mathbb P^2}\xrightarrow{s}\mathcal F(-k)\longrightarrow\mathcal I_Z\otimes\det(\mathcal F(-k))\longrightarrow0.
$$
Since $c_1=-1$, one has
$$
\det(\mathcal F(-k))\simeq\mathcal O_{\mathbb P^2}(-1-2k).
$$
After tensoring by $\mathcal O_{\mathbb P^2}(k)$, we obtain
$$
0\longrightarrow\mathcal O_{\mathbb P^2}(k)\longrightarrow\mathcal F\longrightarrow\mathcal I_Z(-1-k)\longrightarrow0.
$$
Computing the second Chern class, we get
$$
c_2=-k(k+1)+\deg Z.
$$

Since $Z$ is either empty or zero-dimensional, we can choose a line $\ell\subset\mathbb P^2$ such that $\ell\cap Z=\emptyset$. Restricting the exact sequence to $\ell$, we get
$$
0\to \mathcal O_\ell(k)\to \cF|_\ell\to\mathcal O_\ell(-1-k)\to0.
$$
Moreover, $\operatorname{Ext}^1_\ell(\mathcal O_\ell(-1-k),\mathcal O_\ell(k))=H^1(\ell,\mathcal O_\ell(2k+1))=0$. Thus $\cF|_\ell\simeq\mathcal O_\ell(k)\oplus \mathcal O_\ell(-1-k)$. Since $\cF(5/2)$ is nef, we have $\frac32-k\ge0$. Since $k$ is an integer, we obtain $k\le1$. 

If $c_2(\cF)=0$, then $0=-k(k+1)+\deg Z$, so $\deg Z=k(k+1)$. If $k=1$, then $\deg Z=2$. Choose a line $\ell\subset\mathbb P^2$ such that $\ell\cap Z\neq\emptyset$. Restricting the quotient $\mathcal F\twoheadrightarrow\mathcal I_Z(-2)$ to $\ell$, and arguing as in the proof of Lemma~\ref{lem:IZ-nef}, we obtain a quotient line bundle $\mathcal F|_\ell\twoheadrightarrow\mathcal I_{Z\cap\ell/\ell}(-2)$. 

By the nefness of the formal $\mathbb Q$-twist $\mathcal F(5/2)$, we must have $\deg\bigl(\mathcal I_{Z\cap\ell/\ell}(-2)\bigr)+\frac52\geq0$. On the other hand,
$$
\begin{aligned}
\deg\bigl(\mathcal I_{Z\cap\ell/\ell}(-2)\bigr)+\frac52
&=-\deg(Z\cap\ell)-2+\frac52\\
&=\frac12-\deg(Z\cap\ell)<0,
\end{aligned}
$$
which is a contradiction. Hence $k=0$. Then $\deg Z=0$, so $Z=\emptyset$. Thus we obtain
$$
0\to \mathcal O_{\mathbb P^2}\to \cF\to\mathcal O_{\mathbb P^2}(-1)\to0.
$$
Since $\operatorname{Ext}^1(\mathcal O_{\mathbb P^2}(-1),\mathcal O_{\mathbb P^2})=H^1(\mathbb P^2,\mathcal O_{\mathbb P^2}(1))=0$, the sequence splits. Hence $\cF\simeq \mathcal O_{\mathbb P^2}\oplus\mathcal O_{\mathbb P^2}(-1)$.

If $c_2=-2$, then $-2=-k(k+1)+\deg Z$, so $\deg Z=k(k+1)-2$. Since $\deg Z\ge0$ and $k\le1$, we must have $k=1$ and $\deg Z=0$. Thus $Z=\emptyset$ and 
$$
0\to \mathcal O_{\mathbb P^2}(1)\to \cF\to\mathcal O_{\mathbb P^2}(-2)\to0.
$$
Since $\operatorname{Ext}^1(\mathcal O_{\mathbb P^2}(-2),\mathcal O_{\mathbb P^2}(1))=H^1(\mathbb P^2,\mathcal O_{\mathbb P^2}(3))=0$, the sequence splits. Therefore $\cF\simeq \mathcal O_{\mathbb P^2}(1)\oplus \mathcal O_{\mathbb P^2}(-2)$. 
\end{proof}

\section{The case $c_1=0$}
\begin{setting}\label{setting:c1-zero}
Throughout this section, let $\cE$ be a rank $2$ weak Fano bundle on $Q^4$ such that $c_1(\cE)=0$ and $c_2(\cE)=a\alpha+b\beta$. Moreover, we set
$$
D_0:=\xi+2H
$$
on $\mathbb P(\cE)$. By Lemma~\ref{lem:D-nef-big}, $D_0$ is nef and big.
\end{setting}

\subsection{Calculations of intersection numbers}

\begin{lemma}\label{lem:intersections-c1-0}
If $c_1=0$, then
\begin{itemize}
 \item[(i)] $\xi D_0^4=a^2+b^2-24(a+b)+32$,
 \item[(ii)] $HD_0^4=64-8(a+b)$.
\end{itemize}
In particular, $(\xi-2H)D_0^4=a^2+b^2-8(a+b)-96$.
\end{lemma}

\begin{proof}
Recall that for a rank \(2\) vector bundle \(\mathcal E\), the Segre classes are defined by $s_i(\mathcal E)=\pi_*(\xi^{1+i})$. Since \(c_1=0\), we have
\[
\sum_{i\ge0}s_i(\mathcal E)t^i=\frac{1}{1+c_2(\mathcal E)t^2}.
\]
Thus we get $s_0(\mathcal E)=1$, $s_1(\mathcal E)=0$, $s_2(\mathcal E)=-c_2(\mathcal E)$, $s_3(\mathcal E)=0$, and $s_4(\mathcal E)=c_2(\mathcal E)^2$.

Since $D_0=\xi+2H$, we have $\xi D_0^4=\xi(\xi+2H)^4$. Expanding the fourth power gives
$$
(\xi+2H)^4=\xi^4+8H\xi^3+24H^2\xi^2+32H^3\xi+16H^4.
$$
Hence
$$
\xi D_0^4=\xi^5+8H\xi^4+24H^2\xi^3+32H^3\xi^2+16H^4\xi.
$$
Using the projection formula and the above identities for $\pi_*(\xi^i)$, we obtain
$$
\xi D_0^4=c_2(\mathcal E)^2+8H_{Q^4}\cdot 0+24H_{Q^4}^2(-c_2(\mathcal E))+32H_{Q^4}^3\cdot 0+16H_{Q^4}^4.
$$
Therefore
$$
\xi D_0^4=c_2(\mathcal E)^2-24H_{Q^4}^2c_2(\mathcal E)+16H_{Q^4}^4.
$$

Since $c_2(\mathcal E)=a\alpha+b\beta$, $\alpha^2=\beta^2=1$ and $\alpha\beta=0$, we have
$$
c_2(\mathcal E)^2=(a\alpha+b\beta)^2=a^2\alpha^2+2ab\alpha\beta+b^2\beta^2=a^2+b^2.
$$
Moreover, since $H_{Q^4}^2=\alpha+\beta$, we get
$$
H_{Q^4}^2c_2(\mathcal E)=(\alpha+\beta)(a\alpha+b\beta)=a\alpha^2+a\beta\alpha+b\alpha\beta+b\beta^2=a+b.
$$
Substituting these into the previous formula gives
$$
\xi D_0^4=a^2+b^2-24(a+b)+32.
$$
The other formulas are obtained in the same way.
\end{proof}

\subsection{Applications of the splitting criterion}

\begin{proposition}\label{prop:c1-0-negative}
Let $\mathcal E$ be a rank $2$ weak Fano bundle on $Q^4$ with $c_1=0$ and $c_2=(a,b)$. Then the following hold.
\begin{enumerate}
\item If $(a,b)=(-4,-4)$, then $\mathcal E\simeq \mathcal O_{Q^4}(2)\oplus \mathcal O_{Q^4}(-2)$.
\item If $(a,b)=(-1,-1)$, then $\mathcal E\simeq \mathcal O_{Q^4}(1)\oplus \mathcal O_{Q^4}(-1)$.
\item The cases $(a,b)=(-4,0)$, $(0,-4)$, $(-3,-3)$, $(-4,8)$, $(8,-4)$, $(-3,5)$ and $(5,-3)$ do not occur.
\end{enumerate}
\end{proposition}

\begin{proof}
By Lemma~\ref{lem:D-nef-big}, $\mathcal E(3)$ is ample, and hence so is $\mathcal E(4)$. Since $K_{Q^4}\simeq\mathcal O_{Q^4}(-4)$, Theorem~\ref{thm:le} gives
$$
H^i(Q^4,\mathcal E)=H^i\bigl(Q^4,\mathcal E(4)\otimes K_{Q^4}\bigr)=0
$$
for $i\geq2$. For each case in the statement, Lemma~\ref{lem:RR-c1-0} gives $\chi(Q^4, \mathcal E)>0$. Therefore
$$
\chi(Q^4, \mathcal E)=\dim H^0(Q^4,\mathcal E)-\dim H^1(Q^4,\mathcal E)>0,
$$
and hence $H^0(Q^4,\mathcal E)\neq0$. 

First, suppose that $H^0(Q^4,\mathcal E(-2))\neq0$. Then $\xi-2H$ is effective. Since $D_0=\xi+2H$ is nef, we must have $(\xi-2H)D_0^4\geq0$. On the other hand, Lemma~\ref{lem:intersections-c1-0} gives
$$
(\xi-2H)D_0^4=a^2+b^2-8(a+b)-96.
$$
For the numerical cases in the statement, this number is
$$
\begin{cases}
  0
  &\text{if }(a,b)=(-4,-4),\\
  -48
  &\text{if }(a,b)=(-4,0),(0,-4),(-4,8),(8,-4),\\
  -30
  &\text{if }(a,b)=(-3,-3),\\
  -78
  &\text{if }(a,b)=(-3,5),(5,-3),(-1,-1).
\end{cases}
$$
Thus $(a,b)=(-4,-4)$. In this case, Lemma~\ref{lem:intersections-c1-0} also gives
$$
HD_0^4=64-8(a+b)=128.
$$
Consequently,
$$
(\xi-3H)D_0^4=(\xi-2H)D_0^4-HD_0^4=-128<0.
$$
It follows that $H^0(Q^4,\mathcal E(-3))=0$. Set $\cF:=\mathcal E(-2)$. Then $H^0(Q^4,\cF)\neq0$ and $H^0(Q^4,\cF(-1))=0$. Moreover, since $c_1=0$, we have $c_2(\cF)=c_2(\mathcal E)+4H_{Q^4}^2=0$. Therefore, Proposition~\ref{prop:APW-splitting} implies that $\cF$ is a split bundle. Since $c_1(\cF)=-4H_{Q^4}$ and $c_2(\cF)=0$, we obtain $\cF\simeq\mathcal O_{Q^4}\oplus\mathcal O_{Q^4}(-4)$. Hence $\mathcal E\simeq\mathcal O_{Q^4}(2)\oplus\mathcal O_{Q^4}(-2)$.

It remains to consider the case
$$
H^0(Q^4,\mathcal E(-2))=0.
$$
First, suppose that $H^0(Q^4,\mathcal E(-1))\neq0$, and set $\cF:=\mathcal E(-1)$. Then
$$
H^0(Q^4,\cF)\neq0,\qquad H^0(Q^4,\cF(-1))=H^0(Q^4,\mathcal E(-2))=0.
$$
Furthermore,
$$
c_2(\cF)=c_2(\mathcal E)+H_{Q^4}^2=(a+1)\alpha+(b+1)\beta.
$$
For each numerical case in the statement, one has
$$
  a+1<0,
  \qquad\text{or}\qquad
  b+1<0,
  \qquad\text{or}\qquad
  a+1=b+1=0.
$$
Thus, Proposition~\ref{prop:APW-splitting} implies that $\cF$ is a split bundle. Hence $\cE$ is a split bundle.

Finally, suppose that
$$
H^0(Q^4,\mathcal E(-1))=0.
$$
Since $H^0(Q^4,\mathcal E)\neq0$, and since at least one of $a$ and $b$ is negative in every case, Proposition~\ref{prop:APW-splitting} applies directly to $\mathcal E$. Hence $\mathcal E$ is a split bundle.

A rank $2$ split bundle with $c_1=0$ is of the form
$$
\mathcal E\simeq\mathcal O_{Q^4}(m)\oplus\mathcal O_{Q^4}(-m)
$$
for some $m\in\mathbb Z$. Therefore
$$
c_2(\mathcal E)=-m^2H_{Q^4}^2=-m^2\alpha-m^2\beta.
$$
Among the cases in the statement, this is possible only for
$$
  (a,b)=(-4,-4)
  \quad\text{or}\quad
  (a,b)=(-1,-1),
$$
corresponding respectively to $m=2$ and $m=1$. This proves all the assertions.
\end{proof}

\begin{proposition}\label{prop:c1-0-zero}
Let $\mathcal E$ be a rank $2$ weak Fano bundle on $Q^4$ with $c_1=0$ and $c_2=(0,0)$. Then $\mathcal E\simeq \mathcal O_{Q^4}\oplus \mathcal O_{Q^4}$.
\end{proposition}

\begin{proof}
By Lemma~\ref{lem:RR-c1-0}, we have $\chi(Q^4, \cE)=2$. Since $\cE(3)$ is ample, so is $\cE(4)$. Moreover, $K_{Q^4}\simeq\cO_{Q^4}(-4)$. Hence Theorem~\ref{thm:le} gives
$$
H^i(Q^4,\cE)=H^i\bigl(Q^4,\cE(4)\otimes K_{Q^4}\bigr)=0
$$
for $i\ge 2$. Thus $2=\chi(Q^4, \cE)=\dim H^0(Q^4,\cE)-\dim H^1(Q^4,\cE)$, and hence $H^0(Q^4,\cE)\ne0$. We next claim that $H^0(Q^4,\cE(-1))=0$. Indeed, otherwise $\xi-H$ would be effective on $\mathbb P(\cE)$. Since $D_0=\xi+2H$ is nef, we would have $(\xi-H)D_0^4\geq0$. On the other hand, Lemma~\ref{lem:intersections-c1-0} gives
$$
(\xi-H)D_0^4=a^2+b^2-16(a+b)-32=-32,
$$
because $(a,b)=(0,0)$. This is a contradiction.

Thus $H^0(Q^4,\cE)\ne0$ and $H^0(Q^4,\cE(-1))=0$. Since $c_2=(0,0)$, $\cE$ is a split bundle by Proposition~\ref{prop:APW-splitting}. Finally, a split bundle with $c_1=0$ is of the form
$$
\cE\simeq\cO_{Q^4}(m)\oplus\cO_{Q^4}(-m)
$$
for some $m\in\mathbb Z$. Since $c_2=(0,0)$, we have $m=0$. Therefore $\cE\simeq\cO_{Q^4}\oplus\cO_{Q^4}$.
\end{proof}

\subsection{Applications of Lemma~\ref{lem:P2-c1-0-negative}}

\begin{proposition}\label{prop:c1-0-minus-one-three}
There is no rank $2$ weak Fano bundle $\cE$ on $Q^4$ with $c_1=0$ and $c_2=(-1,3)$ or $(3,-1)$.
\end{proposition}

\begin{proof}
We treat the case $c_2=(-1,3)$. The other case is symmetric. Let $\Pi\simeq \P^2$ be an $\alpha$-plane. Then
$$
 c_1(\cE|_{\Pi})=0,
 \qquad
 c_2(\cE|_{\Pi})=-H_{\Pi}^2.
$$
Moreover, $\cE(2)$ is nef. Hence $(\cE|_{\Pi})(2)$ is nef. By Lemma~\ref{lem:P2-c1-0-negative}~(i),
$$
 \cE|_{\Pi}\simeq \cO_{\Pi}(1)\oplus \cO_{\Pi}(-1)
$$
for every $\alpha$-plane $\Pi$. By Lemma~\ref{Hari}, every line \(\ell\subset Q^4\) is contained in a unique \(\alpha\)-plane. Hence
$$
\cE|_\ell\simeq \mathcal O_\ell(1)\oplus \mathcal O_\ell(-1).
$$
Equivalently,
$$
\cE(-1)|_\ell\simeq \mathcal O_\ell\oplus \mathcal O_\ell(-2).
$$
Thus $\cE(-1)$ is linear-uniform of type $2$ in the sense of Definition~\ref{lin}. By Proposition~\ref{prop:Fritzsche}~(ii), $\cE(-1)$ is a split bundle. Hence $\cE$ is a split bundle, contradicting $c_2=(-1,3)$.
\end{proof}

\begin{proposition}\label{prop:c1-0-minus-3-9}
There is no rank $2$ weak Fano bundle $\cE$ on $Q^4$ with $c_1=0$ and $c_2=(-3,9)$ or $(9,-3)$.
\end{proposition}

\begin{proof}
We treat the case $c_2=(-3,9)$. The other case is symmetric. Let $\Pi\simeq \P^2$ be an $\alpha$-plane. Then
$$
 c_1(\cE|_{\Pi})=0,
 \qquad
 c_2(\cE|_{\Pi})=-3H_{\Pi}^2.
$$
Moreover, $\cE(2)$ is nef. Hence $(\cE|_{\Pi})(2)$ is nef. By Lemma~\ref{lem:P2-c1-0-negative}~(ii), no such vector bundle $\cE|_{\Pi}$ exists. Therefore $c_2=(-3,9)$ cannot occur.
\end{proof}

\subsection{Applications of Theorem~\ref{thm:de-cataldo}}

\begin{proposition}\label{prop:c1-0-22}
There is no rank $2$ weak Fano bundle $\cE$ on $Q^4$ with $c_1=0$ and $c_2=(2,2)$.
\end{proposition}

\begin{proof}
Assume that such a vector bundle exists. Since $c_1=0$, the canonical divisor formula gives $-K_{\P(\cE)}=2D_0$, where $D_0=\xi+2H$. By Lemma~\ref{lem:D-nef-big}, $D_0$ is nef and big. Moreover, $D_0-K_{\P(\cE)}=3D_0$ is nef and big. The base-point-free theorem shows that $D_0$ is semiample. Choose a positive integer $m$ such that $|mD_0|$ is base-point-free, and let $\varphi\colon\P(\cE)\to Z$ be the morphism defined by $|mD_0|$, where $Z$ is its image.

\smallskip
\noindent\textsc{Claim.} \emph{Every fiber of $\varphi$ has dimension at most $1$.}

\begin{proof}[Proof of the claim]
Suppose that a fiber of $\varphi$ has dimension at least $2$, and choose an integral surface $S$ contained in that fiber. Since $\cO_{\P(\cE)}(mD_0)\simeq\varphi^*\cO_Z(1)$, the restriction $D_0|_S$ is numerically trivial.

The restriction $\pi|_S\colon S\to\pi(S)$ is finite. Indeed, a positive-dimensional fiber would contain a fiber $\ell\simeq\P^1$ of $\pi$. This is impossible because $D_0\cdot\ell=1$, whereas $D_0|_S\equiv0$. Hence $H|_S$ is ample, being the pullback of an ample divisor by a finite morphism. In particular, $H^2\cdot S>0$.

On the other hand, the Chern--Wu relation on $\P(\cE)$ gives
$$
\xi^2+2H^2=0.
$$
Substituting $\xi=D_0-2H$, we obtain
$$
D_0^2-4D_0H+6H^2=0.
$$
Intersecting this equality with $S$ and using $D_0|_S\equiv0$ yields
$$
0=(D_0^2-4D_0H+6H^2)\cdot S=6H^2\cdot S>0,
$$
a contradiction. This proves the claim.
\end{proof}

Thus the preceding claim and the semiampleness and bigness of $D_0$ allow us to apply Theorem~\ref{thm:de-cataldo} with $\cF=\cE(2)$, $d=4$, $r=2$, $b=1$, and $p=4$. Since $D_0$ is big, we have $\kappa(\cF)=5$, and the inequality in the theorem becomes
$$
4+q>2\bigl[4+(2-1)\bigr]-\min(5,4+2-1)=5,
$$
which is equivalent to $q\geq2$. Since $\Omega_{Q^4}^4\simeq K_{Q^4}\simeq\cO_{Q^4}(-4)$, we obtain
$$
H^q(Q^4,\cE(-2))=0
\qquad\text{for }q\geq2.
$$

By Lemma~\ref{lem:RR-c1-0}, $\chi(Q^4, \cE(-2))=\frac{2\cdot 3+2\cdot 3}{12}=1$. Consequently,
$$
\dim H^0(Q^4,\cE(-2))=1+\dim H^1(Q^4,\cE(-2))\geq1.
$$
Thus $H^0(Q^4,\cE(-2))\ne 0$. Hence $\xi-2H$ is effective. Since $D_0=\xi+2H$ is nef, we must have $(\xi-2H)D_0^4\ge 0$. However, by Lemma~\ref{lem:intersections-c1-0},
$$
 (\xi-2H)D_0^4=a^2+b^2-8(a+b)-96.
$$
Substituting $(a,b)=(2,2)$ gives
$$
 (\xi-2H)D_0^4=8-32-96=-120<0.
$$
This contradicts the nefness of $D_0$.
\end{proof}

Summarizing the results of this section, we obtain the following corollary.

\begin{corollary}\label{cor:c1-0-final}
Let $\mathcal E$ be a normalized rank $2$ weak Fano bundle on $Q^4$ with $c_1=0$. Then $\mathcal E$ is a split bundle. More precisely, $\mathcal E$ is $\mathcal O_{Q^4}\oplus \mathcal O_{Q^4}$, $\mathcal O_{Q^4}(1)\oplus \mathcal O_{Q^4}(-1)$ or $\mathcal O_{Q^4}(2)\oplus \mathcal O_{Q^4}(-2)$.
\end{corollary}

\section{The case $c_1=-1$}

\begin{setting}\label{setting:c1-minus-one}
Throughout this section, let $\cE$ be a rank $2$ weak Fano bundle on $Q^4$ such that $c_1=-1$ and $c_2=(a,b)$. Moreover, we set
$$
D_{-1}:=\xi+\frac{5}{2}H
$$
on $\mathbb P(\cE)$. By Lemma~\ref{lem:D-nef-big}, $D_{-1}$ is nef and big.
\end{setting}

\subsection{Calculations of intersection numbers}
\begin{lemma}\label{lem:intersections-c1-minus}
If $c_1=-1$, then
\begin{itemize}
 \item[(i)] $\xi D_{-1}^4=a^2+b^2-\frac{41}{2}(a+b)+\frac{81}{8}$.
 \item[(ii)] $HD_{-1}^4=68-8(a+b)$.
\end{itemize}
In particular, $(\xi-H)D_{-1}^4=a^2+b^2-\frac{25}{2}(a+b)-\frac{463}{8}$.
\end{lemma}

\begin{proof}
Recall that for a rank $2$ vector bundle $\mathcal E$, the Segre classes
are defined by $s_i(\mathcal E)=\pi_*(\xi^{1+i})$. By the same argument as in Lemma~\ref{lem:intersections-c1-0}, we have
\[
\sum_{i\ge0}s_i(\mathcal E)t^i
=
\frac{1}{
1+H_{Q^4}t+c_2(\mathcal E)t^2
}.
\]
Thus we get $s_1(\cE)=-H_{Q^4}$, $s_2(\cE)=H_{Q^4}^2-c_2(\cE)$, $s_3(\cE)=-H_{Q^4}^3+2H_{Q^4}c_2(\cE)$ and $s_4(\cE)=H_{Q^4}^4-3H_{Q^4}^2c_2(\cE)+c_2(\cE)^2$. By the same argument as in Lemma~\ref{lem:intersections-c1-0}, expanding $\xi(\xi+\frac52H)^4$ and using the projection formula, we obtain
$$
 \xi D_{-1}^4=a^2+b^2-\frac{41}{2}(a+b)+\frac{81}{8}.
$$
The other formulas are obtained in the same way.
\end{proof}

\begin{lemma}\label{lem:c1-minus-vanishing}
Assume that $c_1=-1$. Then
$$
 H^0(Q^4,\cE(-2))=0.
$$
Moreover, if $\chi(Q^4, \cE)>0$, then $H^0(Q^4,\cE)\ne 0$.
\end{lemma}

\begin{proof}
By Lemma~\ref{lem:D-nef-big}, $\cE(3)$ is ample. Since $K_{Q^4}\simeq\cO_{Q^4}(-4)$, Theorem~\ref{thm:le} gives
$$
 H^i(Q^4,\cE(-1))=H^i(Q^4,\cE(3)\tensor K_{Q^4})=0
$$
for $i\ge 2$. By Serre duality, we have
$$
 H^0(Q^4,\cE(-2))^*\simeq H^4(Q^4,(\cE(-2))^*\tensor K_{Q^4})=H^4(Q^4,\cE(-1))=0.
$$
Thus $H^0(Q^4,\cE(-2))=0$. Moreover, Theorem~\ref{thm:le} applied to $\cE(4)$ gives
$$
 H^i(Q^4,\cE)=H^i(Q^4,\cE(4)\tensor K_{Q^4})=0
$$
for $i\ge 2$. Thus $\chi(Q^4, \cE)=\dim H^0(Q^4, \cE)-\dim H^1(Q^4, \cE)$, so $\chi(Q^4, \cE)>0$ implies $H^0(Q^4,\cE)\ne 0$.
\end{proof}

\subsection{Applications of the splitting criterion}
\begin{proposition}\label{prop:c1-minus-a-plus-2-negative}
There is no rank $2$ weak Fano bundle $\mathcal E$ on $Q^4$ with $c_1=-1$ and $c_2=(-3,-3),\ (-3,0),\ (0,-3),\ (-3,1),\ (1,-3),\ (-3,4),\ (4,-3),\ (-3,9)$ or $(9,-3)$.
\end{proposition}

\begin{proof}
In all cases, Lemma~\ref{lem:RR-c1-minus} gives $\chi(Q^4, \cE)>0$. Hence, by Lemma~\ref{lem:c1-minus-vanishing}, $H^0(Q^4,\cE)\ne 0$ and $H^0(Q^4,\cE(-2))=0$. 

If $H^0(Q^4,\cE(-1))\ne 0$, set $\cF=\cE(-1)$. Then $H^0(Q^4,\cF)\ne 0$ and $H^0(Q^4,\cF(-1))=0$. Furthermore, $c_2(\cF)=c_2(\cE(-1))=c_2(\cE)+2H_{Q^4}^2=(a+2)\alpha+(b+2)\beta$. For the cases under consideration, $a+2<0$ or $b+2<0$. Proposition~\ref{prop:APW-splitting} implies that $\cF$ is a split bundle. Hence $\cE$ is a split bundle.

If $H^0(Q^4,\cE(-1))=0$, then Proposition~\ref{prop:APW-splitting} applies directly to $\cE$, since $a<0$ or $b<0$. Thus $\cE$ is a split bundle.

A normalized split bundle with $c_1=-1$ is of the form
\[
\cO_{Q^4}(m)\oplus \cO_{Q^4}(-1-m)
\]
for some $m\in\mathbb Z$. Therefore
\[
c_2(\cE)=-m(m+1)H_{Q^4}^2=-m(m+1)\alpha-m(m+1)\beta.
\]
The listed cases are not of this form. Therefore they do not occur.
\end{proof}

\begin{proposition}\label{prop:c1-minus-minus2}
Let $\mathcal E$ be a rank $2$ weak Fano bundle on $Q^4$. If $c_1=-1$ and $c_2=(-2,-2)$, then $\mathcal E\simeq\mathcal O_{Q^4}(1)\oplus \mathcal O_{Q^4}(-2)$.
\end{proposition}

\begin{proof}
By Lemma~\ref{lem:RR-c1-minus}, $\chi(Q^4, \cE)=6>0$. Hence $H^0(Q^4,\cE)\ne 0$ by Lemma~\ref{lem:c1-minus-vanishing}. We also have $H^0(Q^4,\cE(-2))=0$.

If $H^0(Q^4,\cE(-1))\ne 0$, then $\cF:=\cE(-1)$ satisfies
$$
 H^0(\cF)\ne 0,
 \qquad
 H^0(\cF(-1))=0,
$$
and
$$
 c_2(\cF)=c_2(\cE)+2H_{Q^4}^2=0.
$$
By Proposition~\ref{prop:APW-splitting}, $\cF$ is a split bundle. Hence $\cE$ is a split bundle.

If $H^0(Q^4,\cE(-1))=0$, then Proposition~\ref{prop:APW-splitting} applies directly to $\cE$, since both components of $c_2$ are negative. Hence $\cE$ is a split bundle. The only normalized split bundle with $c_1=-1$ and $c_2=(-2,-2)$ is $\cO_{Q^4}(1)\oplus \cO_{Q^4}(-2)$.
\end{proof}

\subsection{The cases $(-2,3)$, $(3,-2)$ and $(0,0)$}

\begin{proposition}\label{prop:c1-minus-small}
Let $\mathcal E$ be a rank $2$ weak Fano bundle on $Q^4$ with $c_1=-1$. Then there is no such $\mathcal E$ with $c_2=(-2,3)$ or $(3,-2)$. Moreover, if $c_2=(0,0)$, then $\mathcal E\simeq \mathcal O_{Q^4}\oplus \mathcal O_{Q^4}(-1)$.
\end{proposition}

\begin{proof}
For \((-2,3)\), \((3,-2)\), and \((0,0)\), Lemma~\ref{lem:RR-c1-minus} gives \(\chi(Q^4, \cE)=1\). Hence $H^0(Q^4,\cE)\neq0$ by Lemma~\ref{lem:c1-minus-vanishing}. We claim that $H^0(Q^4,\cE(-1))=0$. Indeed, if $H^0(Q^4,\cE(-1))\ne 0$, then $\xi-H$ is effective. Since $D_{-1}=\xi+\frac{5}{2}H$ is nef, we would have
$$
 (\xi-H)D_{-1}^4\ge 0.
$$
But Lemma~\ref{lem:intersections-c1-minus} gives $(\xi-H)D_{-1}^4=a^2+b^2-\frac{25}{2}(a+b)-\frac{463}{8}$. For $(-2,3)$ and $(3,-2)$, this equals
$$
 4+9-\frac{25}{2}-\frac{463}{8}=-\frac{459}{8}<0.
$$
For $(0,0)$, it equals
$$
 -\frac{463}{8}<0.
$$
This is a contradiction. Now Proposition~\ref{prop:APW-splitting} applies to $\cE$. The cases $(-2,3)$ and $(3,-2)$ force splitting, but no normalized split bundle with $c_1=-1$ has either of these second Chern classes. Hence they do not occur. In the case $(0,0)$, the split bundle is $\cO_{Q^4}\oplus \cO_{Q^4}(-1)$. 
\end{proof}

\subsection{Applications of Lemma~\ref{lem:P2-c1-minus-negative}}

\begin{proposition}\label{prop:c1-minus-04}
There is no rank $2$ weak Fano bundle $\mathcal E$ on $Q^4$ with $c_1=-1$ and $c_2=(0,4)$ or $(4,0)$.
\end{proposition}

\begin{proof}
We treat the case $c_2=(0,4)$. The other case is symmetric. Let $\Pi\simeq \P^2$ be an $\alpha$-plane. Then
$$
 c_1(\cE|_{\Pi})=-H_{\Pi},
 \qquad
 c_2(\cE|_{\Pi})=0.
$$
Since $D_{-1}=\xi+\frac52H$ is nef, the formal $\Q$-twist $(\cE|_{\Pi})(5/2)$ is nef. By Lemma~\ref{lem:P2-c1-minus-negative}~(i),
$$
 \cE|_{\Pi}\simeq \cO_{\Pi}\oplus \cO_{\Pi}(-1)
$$
for every $\alpha$-plane $\Pi$. Let $\ell\subset Q^4$ be a line. By Lemma~\ref{Hari}, there is a
unique $\alpha$-plane $\Pi$ containing $\ell$. Therefore
\[
\cE|_\ell
\simeq
\mathcal O_\ell\oplus\mathcal O_\ell(-1).
\]
Thus $\cE$ is linear-uniform of type $1$. By Proposition~\ref{prop:Fritzsche}~(i), \(\cE\) is either a split bundle or a twist of one of the two spinor bundles. Since $c_1=-1$ and the normalized spinor bundles have first Chern class $-1$, this twist is trivial. If $\cE$ is one of the spinor bundles, then $c_2=(1,0)$ or $(0,1)$, which contradicts $c_2=(0,4)$. If $\cE$ is a split bundle, then
$$
c_2=(-m(m+1),-m(m+1))
$$
for some \(m\in\mathbb Z\). This contradicts $c_2=(0,4)$.
\end{proof}

\begin{proposition}\label{prop:c1-minus-additional}
There is no rank $2$ weak Fano bundle $\mathcal E$ on $Q^4$ with $c_1=-1$ and $c_2=(-2,6)$, $(6,-2)$, $(-2,7)$ or $(7,-2)$. 
\end{proposition}

\begin{proof}
We treat the cases $c_2=(-2,6)$ and $(-2,7)$. The cases $c_2=(6,-2)$ and $(7,-2)$ are treated in the same way.

Let $\Pi\subset Q^4$ be an $\alpha$-plane. Then $\Pi\simeq \mathbb P^2$ and 
$$
c_1(\mathcal E|_{\Pi})=-H_{\Pi},\qquad c_2(\mathcal E|_{\Pi})=-2H_{\Pi}^2.
$$
Moreover, since the formal \(\mathbb Q\)-twist $\mathcal E(5/2)$ is nef, its restriction $(\mathcal E|_{\Pi})(5/2)$ is nef. By Lemma~\ref{lem:P2-c1-minus-negative}~(ii), we obtain $\mathcal E|_{\Pi}\simeq\mathcal O_{\Pi}(1)\oplus\mathcal O_{\Pi}(-2)$ for every $\alpha$-plane $\Pi$.

By Lemma~\ref{Hari}, every line $\ell\subset Q^4$ is contained in a unique $\alpha$-plane. Hence, for every line $\ell\subset Q^4$, we have
$$
\mathcal E|_\ell\simeq \mathcal O_\ell(1)\oplus\mathcal O_\ell(-2).
$$
Therefore $\mathcal E(-1)|_\ell\simeq\mathcal O_\ell\oplus\mathcal O_\ell(-3)$. Thus $\mathcal E(-1)$ is linear-uniform of type $3$. By Proposition~\ref{prop:Fritzsche}~(ii), $\mathcal E(-1)$ is a split bundle. Hence $\mathcal E$ is also a split bundle. Since $c_1=-1$, a split bundle $\mathcal E$ is of the form
$$
\mathcal E\simeq
\mathcal O_{Q^4}(m)\oplus
\mathcal O_{Q^4}(-1-m)
$$
for some $m\in\mathbb Z$. Hence
$$
c_2(\mathcal E)=-m(m+1)H_{Q^4}^2=-m(m+1)\alpha-m(m+1)\beta.
$$
In either case, we obtain a contradiction. Hence these cases do not occur.
\end{proof}

\subsection{The remaining cases}

\begin{proposition}\label{prop:c1-minus-remaining}
Let $\mathcal E$ be a rank $2$ weak Fano bundle on $Q^4$ with $c_1=-1$. Assume that $\mathcal E$ is not a split bundle. Then $\cE$ is one of the following:
\begin{enumerate}
\item $c_2=(1,0)$, in which case $\mathcal E$ is the
spinor bundle with $c_2=(1,0)$;
\item $c_2=(0,1)$, in which case $\mathcal E$ is the
spinor bundle with $c_2=(0,1)$;
\item $c_2=(1,1)$, in which case $\mathcal E$ is one of the $H_{Q^4}$-stable bundles described in \cite[Example~2.2]{APW94}.
\end{enumerate}
\end{proposition}

\begin{proof}
By Proposition~\ref{prop:cases} and the preceding propositions, the only remaining cases are $(1,0)$, $(0,1)$, and $(1,1)$. If $H^0(Q^4,\cE)\ne 0$, then $\xi$ would be effective. Since $D_{-1}$ is nef, we have $\xi(\xi+\frac{5}{2}H)^4\ge 0$. But Lemma~\ref{lem:intersections-c1-minus} gives 
$$
\xi \left(\xi+\frac{5}{2}H\right)^4=-\frac{75}{8}
$$
for $c_2=(1,0)$ or $(0,1)$ and
$$
\xi \left(\xi+\frac{5}{2}H\right)^4=-\frac{231}{8}
$$
for $c_2=(1,1)$. Therefore $H^0(Q^4,\cE)=0$. By Proposition~\ref{prop:stable-section}~(i), $\cE$ is $H_{Q^4}$-stable. 

Now we use the classification results in
\cite[Examples~2.1 and~2.2]{APW94}. $H_{Q^4}$-stable bundles of rank $2$ with
$$
 c_1=-1,
 \qquad
 c_2=(1,0)\text{ or }(0,1)
$$
are precisely the two spinor bundles. $H_{Q^4}$-stable bundles of rank $2$ with
$$
 c_1=-1,
 \qquad
 c_2=(1,1)
$$
are precisely the bundles described in \cite[Example~2.2]{APW94}, constructed from two disjoint planes by Serre's construction. These bundles are shown to be Fano bundles in \cite{APW94} and hence are weak Fano bundles.
\end{proof}

Summarizing the results of this section, we obtain the following corollary.

\begin{corollary}\label{cor:c1-minus-final}
Let $\mathcal E$ be a normalized rank $2$ weak Fano bundle on $Q^4$ with $c_1=-1$. Then $\mathcal E$ is $\mathcal O_{Q^4}\oplus \mathcal O_{Q^4}(-1)$, $\mathcal O_{Q^4}(1)\oplus \mathcal O_{Q^4}(-2)$, one of the two spinor bundles on $Q^4$, or one of the $H_{Q^4}$-stable bundles with $c_1=-1$ and $c_2=(1,1)$ described in \cite[Example~2.2]{APW94}.
\end{corollary}

\section{Proof of the main theorem}

\begin{theorem}\label{thm:main}
Let $\cE$ be a rank $2$ weak Fano bundle on a smooth quadric hypersurface $Q^4$. Let \(H_{Q^4}\) denote the hyperplane class on \(Q^4\). Then, up to twisting with a line bundle, $\cE$ is one of the following:
\begin{enumerate}
\item a direct sum of two line bundles;
\item one of the two spinor bundles on $Q^4$;
\item an $H_{Q^4}$-stable rank $2$ bundle with $c_1=-1$ and $c_2=(1,1)$, that is, one of the bundles described in \cite[Example~2.2]{APW94}.
\end{enumerate}
\end{theorem}

\begin{proof}
After tensoring by a line bundle, we may assume that $\cE$ is normalized. If $c_1=0$, then $\cE$ is a split bundle by Corollary~\ref{cor:c1-0-final}. If $c_1=-1$, then Corollary~\ref{cor:c1-minus-final} gives the split bundles, the two spinor bundles, or the $H_{Q^4}$-stable bundles with $c_1=-1$ and $c_2=(1,1)$ described in \cite[Example~2.2]{APW94}. This proves the theorem.
\end{proof}

\subsection*{Acknowledgements}
The author is very grateful to his supervisor, Professor Kiwamu Watanabe, for helpful comments and many valuable discussions for this paper. 

\subsection*{Funding}
The author was partially supported by Japan Science and Technology Agency
(JST) SPRING, Grant Number JPMJSP2170. 

\subsection*{Competing interests}
The author have no competing interests to declare that are relevant to the
content of this article.

\subsection*{Data availability}
Data sharing is not applicable to this article, as no datasets were generated or
analyzed during the current study.

\bibliographystyle{amsalpha}
\bibliography{biblio}

\end{document}